\documentclass[11pt]{amsart}
\usepackage{graphicx}
\usepackage{amssymb}
\usepackage{amscd}
\usepackage{cite}
\usepackage[all]{xy}
\usepackage{url}
\newtheorem{tw}{Theorem}[section]
\newtheorem{prop}[tw]{Proposition}

\theoremstyle{remark}

\theoremstyle{definition}

\newcommand{\Cal}[1]{\mathcal{#1}}

\newcommand{\podz}{\subseteq}

\newcommand{\ro}{\varrho}

\newcommand{\kre}[1]{\overline{#1}}
\newcommand{\gen}[1]{\langle #1 \rangle}
\newcommand{\map}[3]{#1\colon #2\to #3}
\newcommand{\field}[1]{\mathbb{#1}}
\newcommand{\zz}{\field{Z}}

\newcommand{\st}{\;|\;}

\begin{document}

\numberwithin{equation}{section}
\title[The first homology group of the mapping\ldots]
{The first homology group of the mapping class group of a nonorientable surface with twisted coefficients}

\author{Micha\l\ Stukow}

\thanks{Supported by NCN grant 2012/05/B/ST1/02171.}
\address[]{
Institute of Mathematics, University of Gda\'nsk, Wita Stwosza 57, 80-952 Gda\'nsk, Poland }

\email{trojkat@mat.ug.edu.pl}




\begin{abstract}
We determine the first homology group of the mapping class group ${\Cal M}(N)$ of a nonorientable surface $N$ 
with coefficients in $H_1(N;\zz)$.
\end{abstract}

\maketitle%
 \section{Introduction}%
Let $N_{g,s}$ be a smooth, nonorientable, compact surface of genus $g$ with $s$ boundary components. If
$s$ is zero, then we omit it from the notation. If we do not want to emphasise the numbers $g,s$, we simply 
write $N$ for a surface $N_{g,s}$. Recall that $N_{g}$ is a connected sum of $g$ projective planes 
and $N_{g,s}$ is obtained from $N_g$ by removing $s$ open discs.

Let ${\textrm{Diff}}(N)$ be the group of all diffeomorphisms $\map{h}{N}{N}$ such that $h$ is the identity 
on each boundary component. By ${\Cal{M}}(N)$ we denote the quotient group of ${\textrm{Diff}}(N)$ by
the subgroup consisting of maps isotopic to the identity, where we assume that isotopies are 
the identity on each boundary component. ${\Cal{M}}(N)$ is called the \emph{mapping class group} of $N$. 

The mapping class group ${\Cal{M}}(S_{g,s})$ of an orientable surface is defined analogously, but we consider 
only orientation preserving maps. 
\subsection{Background}
Homological computations play a prominent role in the theory of mapping class groups. In the orientable case,
Mumford \cite{MumfordAbel} observed that $H_1({\Cal{M}}(S_g))$ is a quotient of $\zz_{10}$. Then Birman 
\cite{Bir1,Bir1Err} showed that if $g\geq 3$ then $H_1({\Cal{M}}(S_g))$ is a quotient of $\zz_2$, and 
Powell \cite{Powell} showed that in fact $H_1({\Cal{M}}(S_g))$ is trivial if $g\geq 3$. As for 
higher homology groups, Harer \cite{Harer-stab,Harer3} computed $H_i({\Cal{M}}(S_g))$ for $i=2,3$ 
and Madsen and Weiss \cite{MadWeiss} determined the integral cohomology ring of the stable mapping 
class group.

In the nonorientable case, Korkmaz \cite{Kork-non1,Kork-non} computed $H_1({\Cal{M}}(N_g))$ for a 
closed surface $N_g$ (possibly with marked points). This computation was later \cite{Stukow_SurBg} 
extended to the case of a surface with boundary. As for higher homology groups, Wahl 
\cite{Wahl_stab} identified the stable rational cohomology of ${\Cal{M}}(N)$.

As for twisted coefficients, Morita \cite{MoritaJacFou} proved that 
\[H_1({\Cal{M}}(S_g);H_1(S_g;\zz))\cong \zz_{2g-2},\quad \text{for $g\geq 2$.}\]
There are also similar computations for the hyperelliptic mapping class groups. Tanaka 
\cite{Tanaka} showed that $H_1({\Cal{M}}^h(S_g);H_1(S_g;\zz))\cong \zz_{2}$ for $g\geq 2$ and in the nonorientable 
case we showed in \cite{PresHiperNon} that
\[H_1({\Cal{M}}^h(N_g);H_1(N_g;\zz))\cong \zz_{2}\oplus \zz_2\oplus\zz_2,\quad \text{for $g\geq 3$.}\]
\subsection{Main results}
The purpose of this paper is to prove the following theorem.
\begin{tw}\label{MainThm}
If  $N_{g,s}$ is a nonorientable surface of genus $g\geq 3$ with $s\leq 1$ boundary components, then
\[H_1({\Cal{M}}(N_{g,s});H_1(N_{g,s};\zz))\cong \begin{cases}
                                                 \zz_2\oplus\zz_2\oplus\zz_2 &\text{if $g\in\{3,4,5,6\}$}\\
                                                 \zz_2\oplus\zz_2&\text{if $g\geq 7$.}
                                                \end{cases}
\] 
\end{tw}
The starting point for this  computation is the presentation for the mapping class group 
${\Cal{M}}(N_{g,s})$, where $g+s\geq 3$ and $s\in\{0,1\}$, obtained recently by Paris and Szepietowski \cite{SzepParis} 
(Theorems \ref{ParSzep1} and \ref{ParSzep2}).
 \section{Preliminaries}
\subsection{Nonorientable surfaces}
Represent surfaces $N_{g,0}$ and $N_{g,1}$ as respectively a sphere or a disc with $g$ crosscaps and 
let $\alpha_1,\ldots,\alpha_{g-1}$, $\beta_1,\ldots,\beta_{\frac{g-2}{2}}$ be two-sided circles indicated in Figures \ref{r01} and \ref{r02}. 
\begin{figure}[h]
\begin{center}
\includegraphics[width=0.8\textwidth]{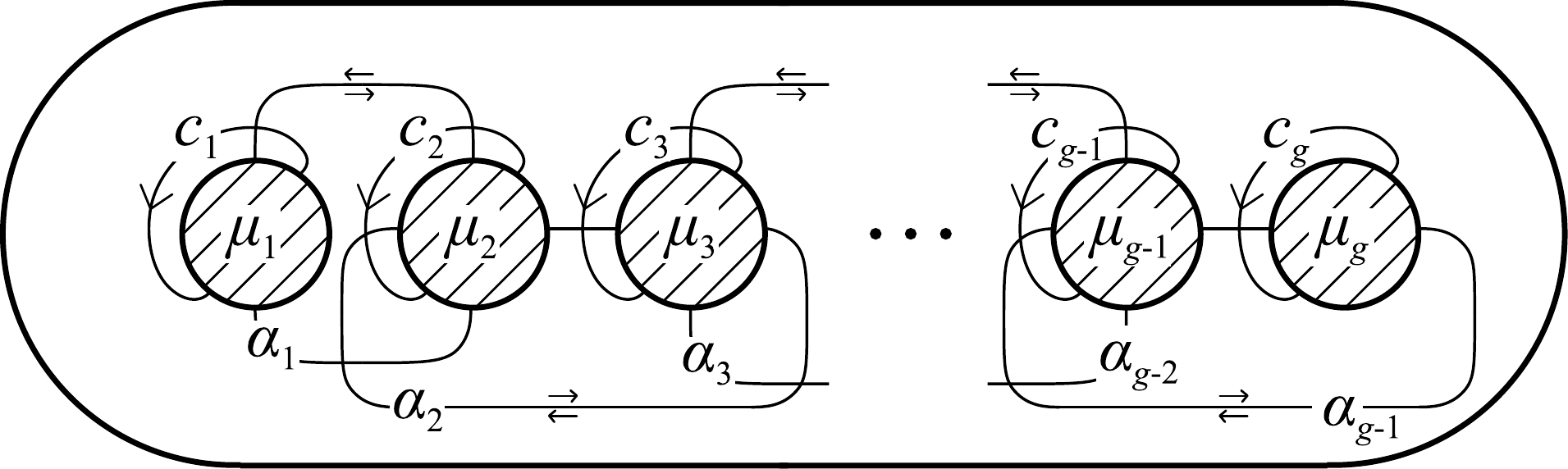}
\caption{Surface $N_{g}$ [$N_{g,1}$] as a sphere [disc] with crosscaps.}\label{r01} %
\end{center}
\end{figure}
\begin{figure}[h]
\begin{center}
\includegraphics[width=0.9\textwidth]{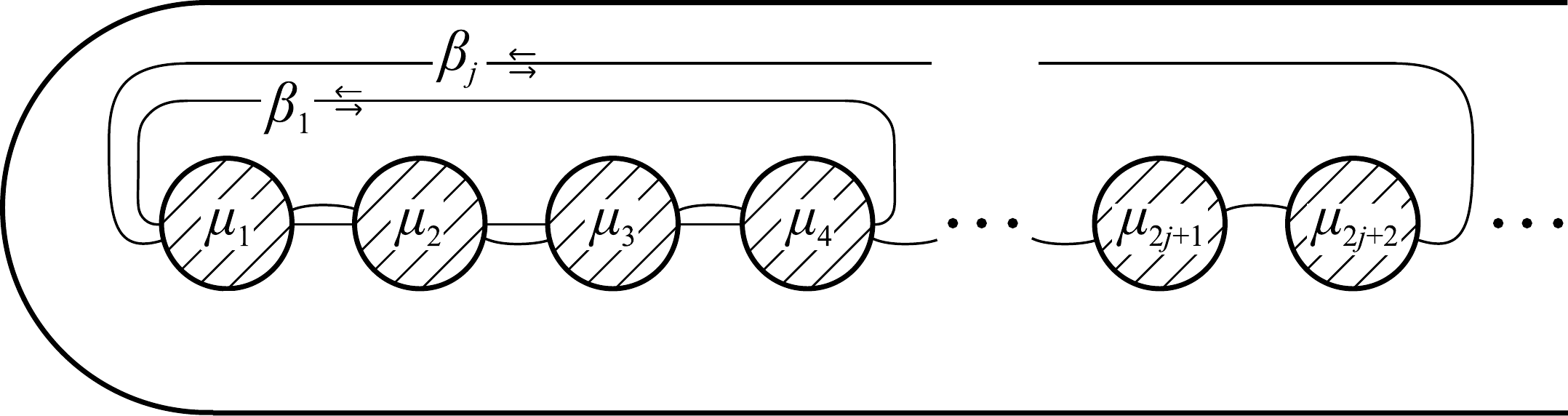}
\caption{Circles $\beta_1,\beta_2,\ldots,\beta_{\frac{g-2}{2}}$.}\label{r02} %
\end{center}
\end{figure}
Small arrows in these figures indicate directions of Dehn twists $a_1,\ldots,a_{g-1}$, $b_1,\ldots,b_{\frac{g-2}{2}}$ associated with these circles. 

For any two consecutive crosscaps $\mu_i,\mu_{i+1}$ we define a \emph{crosscap transposition} $u_i$ 
to be the map which interchanges these two crosscaps (see Figure \ref{r03}).
\begin{figure}[h]
\begin{center}
\includegraphics[width=0.65\textwidth]{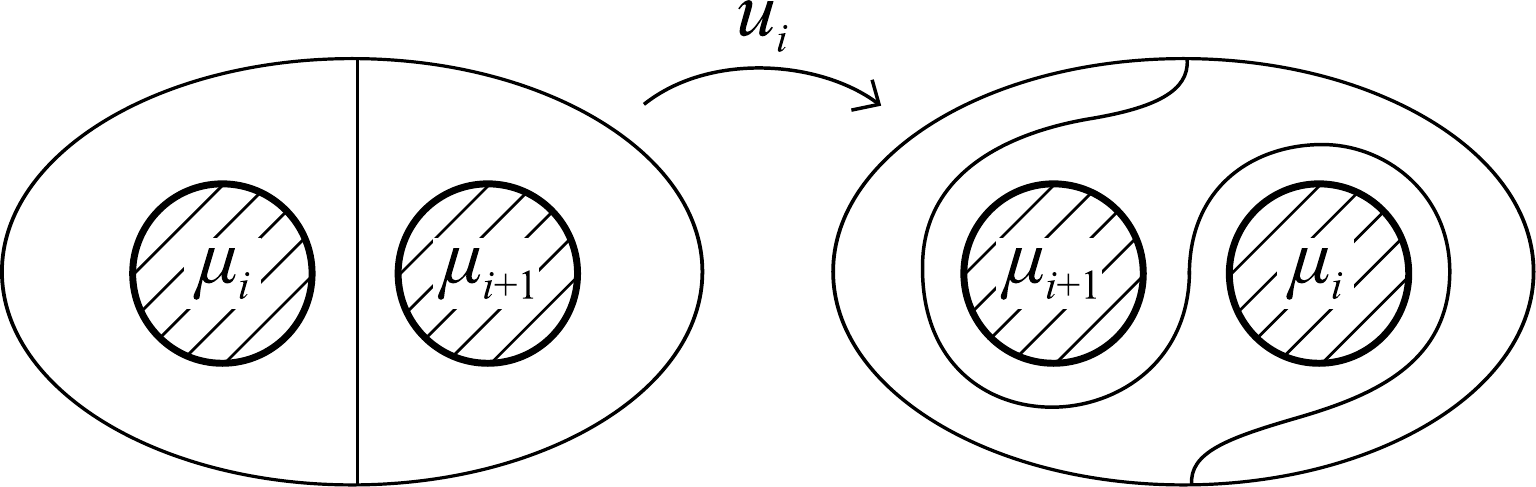}
\caption{Crosscap transposition.}\label{r03} %
\end{center}
\end{figure}
Moreover, if $N_g$ is closed, then we define
\[\ro=a_1a_2\cdots a_{g-1}u_{g-1}\cdots u_2u_{1}.\]
Element $\ro$ represents the \emph{hyperelliptic involution}, that is the reflection of $N_{g}$ 
across the
plane containing the centers of crosscaps in Figure \ref{r01} -- for details see 
\cite{PresHiperNon}.

Recently Paris and Szepietowski \cite{SzepParis} obtained a simple presentation for the mapping 
class group ${\Cal M}(N_{g,s})$ if $g+s>3$ and $s\in\{0,1\}$. The next two theorems state a sightly simplified version 
of Paris-Szpietowski presentation (see Proposition 3.3 and Corollaries 3.2, 3.4 of 
\cite{StukowSimpSzepPar}).
\begin{tw}\label{ParSzep1}
If $g\geq 3$ is odd or $g=4$, then ${\Cal{M}}(N_{g,1})$ admits a presentation with generators:
$a_1,\ldots,a_{g-1},u_1,\ldots,u_{\min\{5,g-1\}}$ and $b_1$ for $g\geq 4$. The defining relations are:
\begin{enumerate}
 \item[(A1)] $a_ja_k=a_ka_j$\quad for $g\geq 4$, $|j-k|>1$,
\item[(A2)] $a_ja_{j+1}a_j=a_{j+1}a_ja_{j+1}$\quad for $j=1,\ldots,g-2$,
 \item[(A3)] $a_jb_1=b_1a_j$\quad for $g\geq 4,j\neq 4$,
 \item[(A4)] $b_1a_4b_1=a_4b_1a_4$\quad for $g\geq 5$,
 \item[(A5)] $(a_2a_3a_4b_1)^{10}=(a_1a_2a_3a_4b_1)^6$\quad for $g\geq 5$,
 \item[(A6)] $(a_2a_3a_4a_5a_6b_1)^{12}=(a_1a_2a_3a_4a_5a_6b_1)^{9}$\quad for $g\geq 7$,
\item[(B1)] $u_1u_3=u_3u_1$\quad for $g\geq 4$,
\item[(B2)] $u_1u_{2}u_1=u_{2}u_1u_{2}$,
\item[(C1)] $u_1a_{j}=a_{j}u_1$\quad for $g\geq 4$, $j=3,\ldots,g-1$,
\item[(C2)] $a_1u_{2}u_1=u_{2}u_1a_{2}$,
\item[(C4)] $a_1u_1a_1=u_1$,
\item[(C5)] $u_{j+1}a_{j}a_{j+1}u_j=a_ja_{j+1}$\quad for $j=1,\ldots,\min\{5,g-1\}$,
\item[(C6)] $(u_3b_1)^2=(a_1a_2a_3)^2(u_1u_2u_3)^2$\quad for $g\geq 4$,
\item[(C7)] $u_5b_1=b_1u_5$\quad for $g\geq 6$,
\item[(C8)] $b_1a_4u_4=a_4u_4(a_4a_3a_2a_1u_1u_2u_3u_4)b_1$\quad for $g\geq 5$.
\end{enumerate}
If $g\geq 6$ is even, then ${\Cal{M}}(N_{g,1})$ admits a presentation with generators:
$a_1,\ldots,a_{g-1}$, $u_1,\ldots,u_{5}$ and additionally: $b_0,b_1,\ldots,b_{\frac{g-2}{2}}$. The defining relations 
are relations: (A1)--(A6), (B1)--(B2), (C1)--(C8) above and additionally:
\begin{enumerate}
 \item[(A7)] $b_0=a_1$,
\item[(A8)] $b_{j+1}(b_{j-1}a_{2j}a_{2j+1}a_{2j+2}a_{2j+3})^{6}=(b_{j-1}a_{2j}a_{2j+1}a_{2j+2}a_{2j+3}b_{j})^5$\quad\\	  for
$1\leq j\leq \frac{g-4}{2}$,
\item[(A9a)] $b_2b_1=b_1b_2$\quad for $g=6$,
\item[(A9b)] $b_{\frac{g-2}{2}}a_{g-5}=a_{g-5}b_{\frac{g-2}{2}}$\quad for $g\geq 8$. 
\end{enumerate}
\end{tw}
\begin{tw}\label{ParSzep2}
 If $g\geq 4$, then the group ${\Cal{M}}(N_{g,0})$ is isomorphic to the quotient of the group ${\Cal{M}}(N_{g,1})$ with the presentation given in Theorem \ref{ParSzep1} obtained by adding a generator $\ro$ and 
 relations:
\begin{enumerate}
\item[(B3)] $(a_1a_2\cdots a_{g-1})^g=
 \begin{cases}
1&\text{for $g$ even}\\
\ro&\text{for $g$ odd,}
\end{cases}$
\item[(D1)] $\ro a_j=a_j\ro$\quad for $j=1,\ldots,g-1$,
\item[(D2)] $u_1\ro u_1=\ro$,
\item[(E)] $\ro^2=1$,
 \item[(F)] $(u_1a_1a_2a_3\cdots a_{g-1}\ro)^{g-1}=1$.
\end{enumerate}
\end{tw}
\subsection{Homology of groups}
Let us briefly review how to compute the first homology of a group with twisted coefficients -- for more details see Section~5 of \cite{PresHiperNon} and references there. 

For a given group $G$ and $G$-module $M$ (that is $\zz G$-module) we define $C_2(G),C_1(G)$ as the free $G$-modules generated
respectively by symbols $[h_1|h_2]$ and $[h_1]$, where $h_i\in G$. We define also $C_0(G)$ as the free $G$-module generated by the
empty bracket $[\cdot]$. Then the first homology group $H_1(G;M)$ is the first homology group of the complex
\[\xymatrix@C=3pc@R=3pc{C_2(G)\otimes_G M\ar[r]^{\ \partial_2\otimes_G {\rm id}}&C_1(G)\otimes_G M
\ar[r]^{\ \partial_1\otimes_G {\rm id}}&C_0(G)\otimes_G M},\]
where 
\[\begin{aligned}
\partial_2([h_1|h_2])&=h_1[h_2]-[h_1h_2]+[h_1],\\
   \partial_1([h])&=h[\cdot]-[\cdot].
  \end{aligned}
\]
For simplicity, we write $\otimes_G=\otimes$ and $\partial\otimes {\rm id}=\kre{\partial}$ henceforth.

If the group $G$ has a presentation $G=\langle X\,|\,R\rangle$ and
\[\gen{\kre{X}}=\gen{[x]\otimes m\st x\in X, m\in M}\podz C_1(G)\otimes M,\]
then $H_1(G;M)$ is a quotient of  $\gen{\kre{X}}\cap \ker\kre{\partial}_1$.

The kernel of this quotient corresponds to relations in $G$
(that is elements of $R$). To be more precise, if
$r\in R$ has the form $x_1\cdots x_k=y_1\cdots y_n$ and $m\in M$, then 
$r$ gives the relation (in $H_1(G;M)$)
\begin{equation}
 \kre{r}\otimes m\!:\ \sum_{i=1}^{k}x_1\cdots x_{i-1}[x_i]\otimes m=\sum_{i=1}^{n}y_1\cdots y_{i-1}[y_i]\otimes m.\label{eq_rew_rel}
\end{equation}
Then 
\[H_1(G;M)=\gen{\kre{X}}\cap \ker\kre{\partial}_1/\gen{\kre{R}},\]
where 
\[\kre{R}=\{\kre{r}\otimes m\st r\in R,m\in M\}.\]
 \section{Action of ${\Cal{M}}(N_{g,s})$ on $H_1(N_{g,s};\zz)$}
Let $c_1,\ldots,c_{g}$ be one-sided circles indicated in Figure \ref{r01}. Recall that if $s\in\{0,1\}$, then the $\zz$-module $H_1(N_{g,s};\zz)$ is generated by $\gamma_1=[c_1],\ldots,\gamma_g=[c_g]$. If $s=1$ then $\gamma_1,\ldots,\gamma_g$ are free generators of $H_1(N_{g,s};\zz)$, and if $s=0$ then they generate $H_1(N_{g,s};\zz)$ with respect to the single relation
\[2(\gamma_1+\gamma_2+\cdots+\gamma_g)=0.\]
The mapping class group ${\Cal{M}}(N_{g,s})$ acts on $H_1(N_{g,s};\zz)$, hence we have a representation
\[\map{\psi}{{\Cal{M}}(N_{g,s})}{\textrm{Aut}(H_1(N_{g,s};\zz))}. \]
It is straightforward to check that
\begin{equation}\begin{aligned}\label{eq:psi}
   \psi(a_j)&=I_{j-1}\oplus \begin{bmatrix}
0&1\\
-1&2\end{bmatrix}\oplus I_{g-j-1}\\
\psi(a_j^{-1})&=I_{j-1}\oplus \begin{bmatrix}
2&-1\\
1&0
    \end{bmatrix}\oplus I_{g-j-1}\\
    \psi(u_j)=\psi(u_j^{-1})&=I_{j-1}\oplus \begin{bmatrix}
0&1\\
1&0
    \end{bmatrix}\oplus I_{g-j-1}\\
    \psi(b_j)&=I_g+
    \begin{bmatrix}-1&1&-1&\ldots&1\\
     -1&1&-1&\ldots&1\\
     \vdots&\vdots&\vdots&\ddots&\vdots\\
     -1&1&-1&\ldots&1\\
    \end{bmatrix}_{2j+2}\oplus 0_{g-(2j+2)}\\
    \psi(b_j^{-1})&=I_g+
    \begin{bmatrix}1&-1&1&\ldots&-1\\
     1&-1&1&\ldots&-1\\
     \vdots&\vdots&\vdots&\ddots&\vdots\\
     1&-1&1&\ldots&-1\\
    \end{bmatrix}_{2j+2}\oplus 0_{g-(2j+2)}\\    
\psi(\ro)&=\psi(\ro^{-1})=-I_g
  \end{aligned}
\end{equation}
where $I_k$ is the identity matrix of rank $k$.
\section{Computing $\gen{\kre{X}}\cap \ker\kre{\partial}_1$}
If $h\in G$, then
\[\kre{\partial}_1([h]\otimes\gamma_i)=(h-1)[\cdot]\otimes\gamma_i=(\psi(h)^{-1}-I_g)\gamma_i,\]
where we identified $C_0(G)\otimes M$ with $M$ by the map $[\cdot]\otimes m\mapsto m$.

Let us denote $[a_j]\otimes \gamma_i, [u_j]\otimes \gamma_i, [\ro]\otimes\gamma_i$ respectively by $a_{j,i},u_{j,i},\ro_i$.
Using formulas \ref{eq:psi}, we obtain
\begin{equation}\begin{aligned}\label{eq:partial}
\kre{\partial}_1(a_{j,i})&=\begin{cases}
                           \gamma_j+\gamma_{j+1}&\text{if $i=j$}\\
                           -\gamma_j-\gamma_{j+1}&\text{if $i=j+1$}\\
                           0&\text{otherwise,}
                          \end{cases}
\\
\kre{\partial}_1(u_{j,i})&=\begin{cases}
                           -\gamma_j+\gamma_{j+1}&\text{if $i=j$}\\
                           \gamma_j-\gamma_{j+1}&\text{if $i=j+1$}\\
                           0&\text{otherwise,}
                          \end{cases}
\\
\kre{\partial}_1(b_{j,i})&=\begin{cases}
                           \gamma_1+\gamma_2+\cdots+\gamma_{2j+2}&\text{if $i=1,3,\ldots,2j+1$}\\
                           -\gamma_1-\gamma_2+\cdots-\gamma_{2j+2}&\text{if $i=2,4,\ldots,2j+2$}\\
                           0&\text{otherwise,}
                          \end{cases}
\\
\kre{\partial}_1(\ro_i)&=-2\gamma_i.
\end{aligned}\end{equation}
\begin{prop}
 Let $g\geq 4$ and $G={\Cal{M}}(N_{g,0})$. Then $\gen{\kre{X}}\cap \ker\kre{\partial}_1$ is the abelian group which admits the presentation with generators:
 \begin{enumerate}
 \item[(G1)] $a_{j,i}$ for $j=1,\ldots,g-1$ and $i=1,\ldots,j-1,j+2,\ldots,g$,
\item[(G2)] $a_{j,j}+a_{j,j+1}$ for $j=1,\ldots,g-1$,
\item[(G3)] $u_{j,i}$ for $j=1,2,\ldots,\min\{5,g-1\}$ and $i=1,\ldots,j-1,j+2,\ldots,g$,
\item[(G4)] $u_{j,j}+u_{j,j+1}$ for $j=1,2,\ldots,\min\{5,g-1\}$,
\item[(G5)] $a_{j,j}-a_{j+1,j+1}+u_{j,j}+u_{j+1,j+1}$ for $j=1,2,\ldots,\min\{5,g-1\}$,
\item[(G6)] $2a_{j,j}+\ro_j+\ro_{j+1}$ for $j=1,\ldots,g-1$,
\item[(G7)] $a_{1,1}+\ro_1=u_{1,1}$,
\item[(G8)] $b_{j,i}$ for $i=2j+3,\ldots,g$,
\item[(G9)] $b_{j,2i}+b_{j,1}$ for $i=1,\ldots,j+1$,
\item[(G10)] $b_{j,2i+1}-b_{j,1}$ for $i=1,\ldots,j$,
\item[(G11)] $b_{j,1}-a_{1,1}-a_{3,3}-\cdots-a_{2j+1,2j+1}$,
\item[(G12)] $\begin{cases}
 a_{1,1}+2a_{2,2}+2a_{4,4}+\cdots+2a_{g-1,g-1}-u_{1,1}&\text{if $g$ is odd}\\
2a_{1,1}+2a_{3,3}+\cdots+2a_{g-1,g-1}&\text{if $g$ is even,}
      \end{cases}$
\end{enumerate}
and relations:
\[\begin{aligned}
r_{a_j}\!&:0=2a_{j,1}+\cdots +2(a_{j,j}+a_{j,j+1})+\cdots+2a_{j,g}\quad\text{for $j=1,\ldots,g-1$}\\
r_{u_j}\!&:0=2u_{j,1}+\cdots +2(u_{j,j}+u_{j,j+1})+\cdots+2u_{j,g}\\ &\qquad\qquad\qquad\qquad\qquad \qquad\qquad\qquad\text{for $j=1,\ldots,\min\{5,g-1\}$}\\
r_{b_j}\!&:0=2(b_{j,2}+b_{j,1})+\cdots+2(b_{j,2j+2}+b_{j,1})\\
&\quad\quad+2(b_{j,3}-b_{j,1})+\cdots+2(b_{j,2j+1}-b_{j,1})+2b_{j,2j+3}+\cdots+2b_{j,g}\\
r_{\ro}\!&:\begin{cases}
             2(a_{1,1}+\ro_1-u_{1,1})+2(2a_{2,2}+\ro_2+\ro_3)+\cdots+2(2a_{g-1,g-1}+\ro_{g-2}+\ro_{g})\\
\qquad\qquad\qquad =2(a_{1,1}+2a_{2,2}+\cdots+2a_{g-1,g-1}-u_{1,1})\quad \text{if $g$ is odd}\\
              2(2a_{1,1}+\ro_1+\ro_2)+\cdots+2(2a_{g-1,g-1}+\ro_{g-1}+\ro_{g})\\
\qquad\qquad\qquad =2(2a_{1,1}+2a_{3,3}+\cdots+2a_{g-1,g-1})\quad \text{if $g$ is even.}
             \end{cases}
\end{aligned}
\]
\end{prop}
\begin{proof}
 By Theorem \ref{ParSzep2}, $\gen{\kre{X}}$ is generated by $a_{j,i},u_{j,i},b_{j,i}$ and $\ro_i$. Using formulas \ref{eq:partial}, it is straightforward to check that elements (G1)--(G12) are elements of $\ker\kre{\partial}_1$. Moreover,
 \[2a_{j,1}+2a_{j,2}+\cdots+2a_{j,g}=[a_{j}]\otimes 2(\gamma_1+\cdots +\gamma_g)=0,\]
hence $r_{a_j}$ is indeed a relation. Similarly we check that $r_{u_j},r_{b_j}$ and $r_{\ro}$ are relations.

Observe that using relations $r_{a_j},r_{u_j},r_{b_j}$ and $r_{\ro}$ we can substitute for $2a_{j,g}$, $2u_{j,g}$, $2b_{j,g}$ and $2\ro_g$ respectively, hence
each element in $\gen{\kre{X}}$ can be written as a linear combination
of $a_{j,i},u_{j,i},b_{j,i},\ro_i$, where each of $a_{j,g},u_{j,g},b_{j,g},\ro_g$ has the coefficient 0 or 1. Moreover, 
for a given $x\in \gen{\kre{X}}\subset C_1(G)\otimes H_1(N_g;\zz)$ such a combination is unique. Hence for the rest of the proof we assume that linear combinations of $a_{j,i},u_{j,i},b_{j,i},\ro_j$ satisfy this condition.

Suppose that $h\in\gen{\kre{X}}\cap\ker\kre{\partial}_1$. We will show that $h$ can be uniquely expressed as a linear combination of generators (G1)--(G12).

We decompose $h$ as follows:
\begin{itemize}
 \item $h=h_0=h_1+h_2$, where $h_1$ is a combination of generators (G1)--(G2) and $h_2$ does not contain $a_{j,i}$ with $i\neq j$;
 \item $h_2=h_3+h_4$, where $h_3$ is a combination of generators (G3)--(G4) and $h_4$ does not contain $u_{j,i}$ with $i\neq j$;
 \item $h_4=h_5+h_6$, where $h_5$ is a combination of generators (G5) and $h_6$ does not contain $u_{j,j}$ for $j>1$;
 \item $h_6=h_7+h_8$, where $h_7$ is a combination of generators (G6) and $h_8$ does not contain $\ro_{j}$ for $j>1$;
 \item $h_8=h_9+h_{10}$, where $h_9$ is a multiple of generator (G7) and $h_{10}$ does not contain $\ro_{1}$;
 \item $h_{10}=h_{11}+h_{12}$, where $h_{11}$ is a combination of generators (G8)--(G10) and $h_{12}$ does not contain $b_{j,i}$ for $i\neq 1$;
 \item $h_{12}=h_{13}+h_{14}$, where $h_{13}$ is a combination of generators (G11) and $h_{14}$ does not contain $b_{j,i}$.
\end{itemize}
Observe also that for each $k=0,\ldots,12$, $h_{k+1}$ and $h_{k+2}$ are uniquely determined by $h_k$. Element $h_{14}$ has the form \[h_{14}=\sum_{j=1}^{g-1}\alpha_ja_{j,j}+\alpha u_{1,1}\]
for some integers $\alpha,\alpha_1,\ldots,\alpha_{g-1}$. Hence
\[0=\kre{\partial}_1(h_{14})=(\alpha_1-\alpha)\gamma_1+(\alpha_1+\alpha_2+\alpha)\gamma_2+(\alpha_2+\alpha_3)\gamma_3+\cdots+\alpha_{g-1}\gamma_g.\]
If $g$ is odd this implies that 
\[\alpha_2=\alpha_4=\cdots=\alpha_{g-1}=2k,\ \alpha_3=\alpha_5=\cdots=\alpha_{g-2}=0,\ \alpha_1=k,\ \alpha=-k\]
for some $k\in\zz$. For $g$ even we get 
\[\alpha_1=\alpha_3=\alpha_5=\cdots=\alpha_{g-1}=2k,\ \alpha_2=\alpha_4=\cdots=\alpha_{g-2}=0,\ \alpha=0.\]
In both cases $h_{14}$ is a multiple of the generator (G12).
\end{proof}
By an analogous argument we get
\begin{prop}
 Let $g\geq 3$ and $G={\Cal{M}}(N_{g,1})$. Then $\gen{\kre{X}}\cap \ker\kre{\partial}_1$ is the abelian group generated freely by:
 \begin{enumerate}
 \item[(G1)] $a_{j,i}$ for $j=1,\ldots,g-1$ and $i=1,\ldots,j-1,j+2,\ldots,g$,
\item[(G2)] $a_{j,j}+a_{j,j+1}$ for $j=1,\ldots,g-1$,
\item[(G3)] $u_{j,i}$ for $j=1,2,\ldots,\min\{5,g-1\}$ and $i=1,\ldots,j-1,j+2,\ldots,g$,
\item[(G4)] $u_{j,j}+u_{j,j+1}$ for $j=1,2,\ldots,\min\{5,g-1\}$,
\item[(G5)] $a_{j,j}-a_{j+1,j+1}+u_{j,j}+u_{j+1,j+1}$ for $j=1,2,\ldots,\min\{5,g-1\}$,
\item[(G8)] $b_{j,i}$ for $i=2j+3,\ldots,g$,
\item[(G9)] $b_{j,2i}+b_{j,1}$ for $i=1,\ldots,j+1$,
\item[(G10)] $b_{j,2i+1}+b_{j,1}$ for $i=1,\ldots,j$,
\item[(G11)] $b_{j,1}-a_{1,1}-a_{3,3}-\cdots-a_{2j+1,2j+1}$.
\end{enumerate}
\end{prop}
 \section{Computing $H_1({\Cal{M}}(N_{g,s});H_1(N_{g,s};\zz))$}
Using formula \eqref{eq_rew_rel} we rewrite relations from Theorems \ref{ParSzep1} and \ref{ParSzep2} as relations in 
$H_1({\Cal{M}}(N_{g,s});H_1(N_{g,s};\zz))$
\subsection*{(A1)--(A2)}
Relation (A1) 
gives 
\[\begin{aligned}r^{(A1)}_{j,k:i}\!:0&=([a_j]+a_j[a_k]-[a_k]-a_k[a_j])\otimes \gamma_i\\
&=a_{j,i}+[a_k]\otimes \psi(a_j^{-1})\gamma_i-a_{k,i}-[a_j]\otimes \psi(a_k^{-1})\gamma_i\\
&=\pm\begin{cases}
   0&\text{if $i\not\in\{ j,j+1,k,k+1\}$}\\
a_{k,j}+a_{k,j+1}&\text{if $i\in\{j,j+1\}$}\\
a_{j,k}+a_{j,k+1}&\text{if $i\in\{k,k+1\}$.}
  \end{cases}
  \end{aligned}
\]
Relation (A2) gives
\[\begin{aligned}r^{(A2)}_{j:i}\!:0=&
([a_j]+a_{j}[a_{j+1}]+a_{j}a_{j+1}[a_{j}]\\
&-[a_{j+1}]-a_{j+1}[a_{j}]-a_{j+1}a_{j}[a_{j+1}])\otimes \gamma_i\\
=&[a_j]\otimes (I_g+\psi(a_{j+1}^{-1}a_j^{-1})-\psi(a_{j+1}^{-1}))\gamma_i\\
&+[a_{j+1}]\otimes(\psi(a_j^{-1})-I_g-\psi(a_j^{-1}a_{j+1}^{-1}))\gamma_i\\
=&\begin{cases}
                 a_{j,i}-a_{j+1,i}&\text{\hspace{-1.5cm}if $i\not\in\{j,j+1,j+2\}$}\\
 a_{j,j+2}-a_{j+1,j}&\text{\hspace{-1.5cm}if $i=j+2$}\\
(*)+2(a_{j,j}+a_{j,j+1})&\text{\hspace{-1.5cm}if $i=j$}\\
(*)-(a_{j,j}+a_{j,j+1})-(a_{j+1,j+1}+a_{j+1,j+2})&\text{if $i=j+1$.}
                 \end{cases}
\end{aligned}\]
In the above formula $(*)$ denotes some expression homologous to 0 by previously obtained relations. As we progress further, we will often perform simplifications based on previously obtained relations, from now on we will use symbol '$\equiv$' in such cases.

Carefully 
checking relations $r^{(A1)}_{j,k:i}$ and $r^{(A2)}_{j:i}$ we conclude that generators (G1) generate a cyclic group,
and generators (G2) generate a cyclic group of order at most 2.
\subsection*{(C4)}
Relation (C4) gives
\[\begin{aligned}
   r^{(C4)}_{i}\!:0=&([a_1]+a_1[u_1]+a_1u_1[a_1]-[u_1])\otimes\gamma_i\\
   =&a_{1,i}+[u_1]\otimes\psi(a_1^{-1})\gamma_i+[a_1]\otimes \psi(u_1^{-1}a_1^{-1})\gamma_i-u_{1,i}\\
   =&\pm\begin{cases}
      (u_{1,1}+u_{1,2})+2(a_{1,1}+a_{1,2})&\text{for $i=1$}\\
      u_{1,1}+u_{1,2}&\text{for $i=2$}\\
      2a_{1,i}=0&\text{for $i>2$.}
     \end{cases}
  \end{aligned}
\]
Hence the cyclic group generated by generators (G1) has order at most 2 and generator (G4) is trivial for $j=1$.
\subsection*{(C1)}
Relation (C1) gives
\[\begin{aligned}
   r^{(C1)}_{j:i}\!:0=&([u_1]+u_1[a_j]-[a_j]-a_j[u_1])\otimes \gamma_i\\
   =&u_{1,i}+[a_j]\otimes\psi(u_1^{-1})\gamma_i-a_{j,i}-[u_1]\otimes \psi(a_j^{-1})\gamma_i\\
   =&\pm \begin{cases}
      a_{j,1}-a_{j,2}&\text{for $i\in\{1,2\}$}\\
      u_{1,j}+u_{1,j+1}&\text{for $i\in\{j,j+1\}$}\\
      0&\text{for $i\not\in\{ 1,2,j,j+1\}$.}
     \end{cases}
  \end{aligned}
\]
The above relation implies that generators (G3) of the form $u_{1,j}$ where $j\geq 3$ generate a cyclic group (note that this statement holds also for $g=3$).
\subsection*{(C5)}
Relation (C5) gives 
\[\begin{aligned}
   r^{(C5)}_{j:i}\!:0=&([u_{j+1}]+u_{j+1}[a_j]+u_{j+1}a_j[a_{j+1}]+u_{j+1}a_ja_{j+1}[u_j]-[a_j]-a_j[a_{j+1}])\otimes\gamma_i\\
   =&u_{j+1,i}+[a_j]\otimes(\psi(u_{j+1}^{-1})-I_g)\gamma_i+[a_{j+1}]\otimes(\psi(a_j^{-1}u_{j+1}^{-1})-\psi(a_j^{-1}))\gamma_i\\
   &+[u_j]\otimes \psi(a_{j+1}^{-1}a_j^{-1}u_{j+1}^{-1})\gamma_i\\
   &\equiv\begin{cases}
      u_{j+1,i}+u_{j,i}&\text{if $i\not\in\{j,j+1,j+2\}$}\\
      (u_{j+1,j}+u_{j,j+2})+2(u_{j,j}+u_{j,j+1})&\text{if $i=j$}\\
      (a_{j,j}-a_{j+1,j+1}+u_{j,j}+u_{j+1,j+1})\\
      \quad\ -(u_{j,j}+u_{j,j+1})&\text{if $i=j+1$}\\
            -(a_{j,j}-a_{j+1,j+1}+u_{j,j}+u_{j+1,j+1})\\
      \quad\ +(u_{j+1,j+1}+u_{j+1,j+2})&\text{if $i=j+2$.}
     \end{cases}
   \end{aligned}\]
   From the last two equalities we obtain that generators (G4) generate a cyclic group, which by relation (C4) is trivial. Hence generators (G5) are also trivial.
   
   Now we are going to show that generators (G3) generate a cyclic group.
   
   If $k<i-1$ then from the first equality and relation (C1) we have
   \[u_{k,i}=-u_{k-1,i}=\cdots=\pm u_{1,i}=\pm u_{1,3}.\]
   As for the elements $u_{k,i}$ with $k>i$, by the the first two equalities, 
   \[\begin{aligned}
      u_{5,4}&=-u_{4,6}\\
      u_{5,3}&=-u_{4,3}=u_{3,5}\\
      u_{5,2}&=-u_{4,2}=u_{3,2}=-u_{2,4}\\
      u_{5,1}&=-u_{4,1}=u_{3,1}=-u_{2,1}=u_{1,3}.
     \end{aligned}
\]
Hence in fact generators (G3) generate a cyclic group.
\subsection*{(B2)}
Relation (B2) gives
\[\begin{aligned}r^{(B2)}_{i}\!:0&=
([u_1]+u_{1}[u_{2}]+u_{1}u_{2}[u_{1}]
-[u_{2}]-u_{2}[u_{1}]-u_{2}u_{1}[u_{2}])\otimes \gamma_i\\
&=\begin{cases}
                 u_{1,i}-u_{2,i}&\text{if $i>3$}\\
 u_{1,3}-u_{2,1}&\text{if $i\in\{1,3\}$}\\
(u_{1,1}+u_{1,2})-(u_{2,2}+u_{2,3})-(u_{1,3}-u_{2,1})&\text{if $i=2$.}
                 \end{cases}
\end{aligned}\]
In particular, $u_{1,3}=u_{2,1}$. But relation (C5) implies that $u_{1,3}=-u_{2,1}$. Hence the cyclic group generated by generators (G3) has order at most~2.
\subsection*{(B1)}
Relation (B1) gives
\[\begin{aligned}r^{(B1)}_{i}\!:0&=([u_1]+u_1[u_3]-[u_3]-u_3[u_1])\otimes \gamma_i\\
&=\pm\begin{cases}
   0&\text{if $i>4$}\\
   u_{3,2}-u_{3,1}&\text{if $i\in\{1,2\}$}\\
   u_{1,4}-u_{1,3}&\text{if $i\in\{3,4\}$.}
  \end{cases}
  \end{aligned}
\]
This relation gives no new information.
\subsection*{(C2)}
Relation (C2) gives
\[\begin{aligned}
   r^{(C2)}_{i}\!:0&=([a_1]+a_1[u_2]+a_1u_2[u_1]-[u_2]-u_2[u_1]-u_2u_1[a_2])\otimes \gamma_i\\
   &=\pm \begin{cases}
   a_{1,i}-a_{2,i}&\text{if $i>3$}\\
   (a_{1,1}-a_{2,2}+u_{1,1}+u_{2,2})+(u_{1,3}+u_{2,1})&\text{if $i=1$}\\
   (a_{1,1}+a_{1,2})-(a_{2,2}+a_{2,3})\\
   \quad -(a_{1,1}-a_{2,2}+u_{1,1}+u_{2,2})-(u_{1,3}+u_{2,1})&\text{if $i=2$}\\
   a_{1,3}-a_{2,1}&\text{if $i=3$.}
  \end{cases}
\end{aligned}\]
This relation gives no new information.

Note that at this point we proved Theorem \ref{MainThm} for $N=N_{3,1}$,
\[\begin{aligned}H_1({\Cal{M}}(N_{3,1});H_1(N_{3,1};\zz))&\cong \gen{a_{1,3},a_{1,1}+a_{1,2},u_{1,3}}
\cong \zz_2\oplus\zz_2\oplus\zz_2.\end{aligned}\]
Observe also that for $N=N_{3,0}$, Theorem \ref{MainThm} follows from Theorem 5.4 of \cite{PresHiperNon}, hence from now we assume that $g\geq 4$.

\subsection*{(A3)} Relation (A3) gives
\[\begin{aligned}
    r^{(A3)}_{j:i}\!:0&=([a_j]+a_j[b_1]-[b_1]-b_1[a_j])\otimes \gamma_i\\
   &=a_{j,i}+[b_1]\otimes \psi(a_j^{-1})\gamma_i-b_{1,i}-[a_j]\otimes \psi(b_1^{-1})\gamma_i.
  \end{aligned}
\]
If $i\neq j$ and $i\neq j+1$, then we get
\[0\equiv\begin{cases}
(a_{1,1}+a_{1,2})+(a_{1,3}+a_{1,4})&\text{if $j=1$ and $i\in\{3,4\}$}\\
(a_{2,2}+a_{2,3})+(a_{2,1}+a_{2,4})&\text{if $j=2$ and $i\in\{1,4\}$}\\
(a_{3,3}+a_{3,4})+(a_{3,1}+a_{3,2})&\text{if $j=3$ and $i\in\{1,2\}$}\\
   0&\text{if $j>4$ or $i>4$}
  \end{cases}\]
By relation $r^{(A1)}$, this implies that generators (G2) are trivial for $g\geq 4$.

For $i=j$ and $i=j+1$ we get
\[0=\pm\begin{cases}
     (b_{1,1}+b_{1,2})-(a_{1,1}+a_{1,2})-(a_{1,3}+a_{1,4})&\text{if $j=1$}\\
     (b_{1,1}+b_{1,2})+(b_{1,3}-b_{1,1})+(a_{2,2}+a_{2,3})+(a_{2,1}+a_{2,4})&\text{if $j=2$}\\
     (b_{1,3}-b_{1,1})+(b_{1,1}+b_{1,4})-(a_{3,3}+a_{3,4})-(a_{3,1}+a_{3,2})&\text{if $j=3$}\\
     b_{1,j}+b_{1,j+1}&\text{if $j>4$.}
    \end{cases}
\]
This implies that generators (G8) for $j=1$ generate a cyclic group, and generators (G9)--(G10) are trivial for $j=1$.
\subsection*{(A4)}
Relation (A4) gives
\[\begin{aligned}
   r^{(A4)}_{i}\!:0&=([b_1]+b_1[a_4]+b_1a_4[b_1]-[a_4]-a_4[b_1]-a_4b_1[a_4])\otimes \gamma_i=\\
   &=b_{1,i}+[a_4]\otimes \psi(b_1^{-1})\gamma_i+[b_1]\otimes\psi(a_4^{-1}b_1^{-1})\gamma_i-a_{4,i}\\
   &\quad -[b_1]\otimes \psi (a_4^{-1})\gamma_i-[a_4]\otimes\psi(b_1^{-1}a_4^{-1})\gamma_i\\
  &\equiv\begin{cases}
   \pm(b_{1,5}-a_{4,i})&\text{if $i\in\{1,2,3\}$}\\
   -b_{1,5}-(a_{4,4}+a_{4,5})+(a_{4,1}+a_{4,2}+a_{4,3})
   &\text{if $i=4$}\\
   b_{1,5}-a_{4,1}-a_{4,2}-a_{4,3}&\text{if $i=5$}\\   
   b_{1,i}-a_{4,i}&\text{if $i>5$.}\\
  \end{cases}
  \end{aligned}
\]
This implies that generators (G8) for $j=1$ are in the cyclic group generated by generators (G1).
\subsection*{(C6) and simplification trick}
Relation (C6) gives
\[\begin{aligned}
   r^{(C6)}_{i}\!:0&=(1+u_3b_1)([u_3]+u_3[b_1])\otimes\gamma_i-(1+a_1a_2a_3)([a_1]+a_1[a_2]+a_1a_2[a_3])\otimes\gamma_i\\
   &\quad-(a_1a_2a_3)^2(1+u_1u_2u_3)([u_1]+u_1[u_2]+u_1u_2[u_3])\otimes\gamma_i
  \end{aligned}
\]
For $i>4$ this easily yields
\[2b_{1,i}-2(a_{1,i}+a_{2,i}+a_{3,i})-2(u_{1,i}+u_{2,i}),\]
which gives no new information.

The computations for $i\in \{1,2,3\}$ are a bit cumbersome, hence before we go into the details, we make a couple of observations which will greatly simplify the situation.
\begin{itemize}
 \item Observe first that if we express $r^{(C6)}_{i}$, $i\in\{1,2,3\}$ as a combination of generators (G1)--(G5),(G8)--(G12), as we did before for other relations, we will obtain an expression without generators (G12). The reason for this is that relation (C6) does not depend on $g$, hence the resulting expression will be the same for all $g\geq 4$ and it will not contain $a_{j,j}$ with $j\geq 4$.
 \item In the obtained expression for $r^{(C6)}_{i}$, $i\in\{1,2,3\}$ we can perform the following changes, which are consequences of previously obtained relations.
 \begin{itemize}
 \item We replace each generator (G1) with $a_{1,3}$ and each generator (G3) with $u_{1,3}$. Moreover, since $2a_{1,3}=0$, we can assume that $a_{1,3}$ has the coefficient 0 or 1. The same for $u_{1,3}$.
 \item Since generators (G2) and (G4) are trivial, we can remove all elements $a_{j,j+1}$ and $u_{j,j+1}$ (we replace them respectively by $-a_{j,j}$ and $-u_{j,j}$).
 \item Since generators (G5) are trivial, we remove elements $u_{j,j}$ for $j>1$. Observe that the obtained expression automatically does not contain $u_{1,1}$.
 \item Since generators (G9)-(G10) are trivial for $j=1$, we replace elements $b_{1,i}$ for $i>1$ with $\pm b_{1,1}$.
 \end{itemize}
 \item As the result of the above changes we obtain an expression for $r^{(C6)}_{i}$ which involves only $a_{1,3}$, $u_{1,3}$ and generators (G8), (G11). Observe also that coefficients of generators (G11) are uniquely determined by coefficients of corresponding $b_{j,1}$.
\end{itemize}
The above analysis implies that during the transformation of $r^{(C6)}_{i}$ we can completely ignore elements $a_{j,j},a_{j,j+1},u_{j,j},u_{j,j+1}$, and the elements $a_{j,i}, u_{j,i}$ with $i\not\in\{j,j+1\}$ are subject to the equivalence
 $\pm a_{j,i}\equiv a_{1,3}$, $\pm u_{j,i}\equiv u_{1,3}$. In particular we can ignore all elements of the form
 \[[a_j]\otimes 2m,\ [u_j]\otimes 2m,\ \text{for }m\in H_1(N_g;\zz)\]
and as a result
\[\begin{aligned}
   \, [a_j]&\otimes \psi(a_k^{\pm 1})\gamma_i\equiv[a_j]\otimes \psi(u_k)\gamma_i,\\
   [u_j]&\otimes \psi(a_k^{\pm 1})\gamma_i\equiv [u_j]\otimes \psi(u_k)\gamma_i.\\
  \end{aligned}
\]
Consequently, if $w$ is any word in letters 
\[\{a_1^{\pm 1},\ldots,a_{g-1}^{\pm 1},u_1^{\pm 1},\ldots,u_{g-1}^{\pm 1},b_{1}^{\pm 1}\}\]
and $w'$ is the word obtained from $w$ by replacing each $a_k^{\pm 1}$ with $u_k$
then
\[\begin{aligned}
   \ [a_j]&\otimes \psi (w)\gamma_i\equiv [a_j]\otimes \psi (w')\gamma_i,\\
   [u_j]&\otimes \psi (w)\gamma_i\equiv [u_j]\otimes \psi (w')\gamma_i.
  \end{aligned}
\]
Returning to the relation $r^{(C6)}_{i}$, it is equivalent to 
\[\begin{aligned}
  r^{(C6)}_{i}\!:0&=([u_3]+u_3[b_1])\otimes(I_g+\psi(b_1^{-1}u_3))\gamma_i\\
  &-([a_1]+u_1[a_2]+u_1u_2[a_3])\otimes(I_g+\psi(u_3u_2u_1))\gamma_i\\
  &-([u_1]+u_1[u_2]+u_1u_2[u_3])\otimes(\psi((u_3u_2u_1)^2)+\psi((u_3u_2u_1)^3))\gamma_i\\
  &\equiv \begin{cases}
     \pm 2(b_{1,1}-a_{1,1}-a_{3,3})&\text{if $i\in\{1,2\}$}\\
     0&\text{if $i\in\{3,4\}$.}
    \end{cases}
  \end{aligned}
\]
Hence generator (G11) has order 2 for $j=1$. In particular, as in the case of $[a_{j}]$ and $[u_j]$, we can ignore elements of the form 
\[[b_1]\otimes 2m,\ \text{for }m\in H_1(N_g;\zz).\]
Consequently,
\[[b_1]\otimes \psi (w)\gamma_i\equiv[b_1]\otimes \psi (w')\gamma_i,\]
where $w'$ is the word obtained from $w$ by replacing each $a_k^{\pm 1}$ with $u_k$.

Note that at this point we have proved Theorem \ref{MainThm} for $N=N_{4,1}$, 
\[\begin{aligned}H_1({\Cal{M}}(N_{4,1});H_1(N_{4,1};\zz))&\cong \gen{a_{1,3},u_{1,3},b_{1,1}-a_{1,1}-a_{3,3}}
\cong \zz_2\oplus\zz_2\oplus\zz_2.\end{aligned}\]
\subsection*{(C7)}
Relation (C7) gives
\[\begin{aligned}
   r^{(C7)}_{i}\!:0&=([u_5]+u_5[b_1]-[b_1]-b_1[u_5])\otimes \gamma_i\\
   &=[u_5]\otimes(I_g-\psi(b_1^{-1}))\gamma_i-[b_1]\otimes(I_g-\psi(u_5))\gamma_i\\
   &=\pm\begin{cases}
   0&\text{if $i>6$}\\
      b_{1,5}-b_{1,6}&\text{if $i\in \{5,6\}$}\\
      u_{5,1}+u_{5,2}+u_{5,3}+u_{5,4}&\text{if $i\in \{1,2,3,4\}$.}
     \end{cases}
  \end{aligned}
\]
 This relation gives no new information.
\subsection*{(C8)}
Since $u_j^2=id$ as an automorphism of $H_1(N;\zz)$, relation (C8) gives
\[\begin{aligned}
   r^{(C8)}_i\!:0&=([a_4]+a_4[u_4]+a_4u_4[a_4]+a_4u_4a_4[a_3]+a_4u_4a_4a_3[a_2])\otimes\gamma_i\\
   &\ +(a_4u_4a_4a_3a_2[a_1]+a_4u_4a_4a_3a_2a_1[u_1]+
   a_4u_4a_4a_3a_2a_1u_1[u_2])\otimes\gamma_i\\
   &\ +(a_4u_4a_4a_3a_2a_1u_1u_2[u_3]+a_4u_4a_4a_3a_2a_1u_1u_2u_3[u_4])\otimes\gamma_i\\
   &\ +(a_4u_4a_4a_3a_2a_1u_1u_2u_3u_4[b_1]-
   [b_1]-b_1[a_4]-b_1a_4[u_4])\otimes\gamma_i\\
   &=([a_4]+u_4[u_4]+[a_4]+u_4[a_3]+u_4u_3[a_2]+u_4u_3u_2[a_1])\otimes\gamma_i\\
   &\ +(u_4u_3u_2u_1[u_1]+
   u_4u_3u_2[u_2]+u_4u_3[u_3]+u_4[u_4])\otimes\gamma_i\\   
   &\ +([b_1]-
   [b_1]-b_1[a_4]-b_1u_4[u_4])\otimes\gamma_i\\
   &=(u_4[a_3]+u_4u_3[a_2]+u_4u_3u_2[a_1]+u_4u_3u_2u_1[u_1])\otimes\gamma_i\\
   &\ +(u_4u_3u_2[u_2]+u_4u_3[u_3]-b_1[a_4]-b_1u_4[u_4])\otimes\gamma_i.
  \end{aligned}
\]
Now it is straightforward to check that this relation gives no new information.
\subsection*{(A5)}
Relation (A5) gives
\[\begin{aligned}
   r^{(A5)}_{i}\!:0&=\left(\sum_{p=0}^{9}(a_2a_3a_4b_1)^p\right)([a_2]+a_2[a_3]+a_2a_3[a_4]+a_2a_3a_4[b_1])\otimes\gamma_i\\
   &\ -\left(\sum_{p=0}^{5}(a_1a_2a_3a_4b_1)^p\right)([a_1]+a_1[a_2]+a_1a_2[a_3]+a_1a_2a_3[a_4]\\
   &\quad\quad\quad+a_1a_2a_3a_4[b_1])\otimes\gamma_i\\
   &=([a_2]+a_2[a_3]+a_2a_3[a_4]+a_2a_3a_4[b_1])\otimes\left(\sum_{p=0}^{9}A^p\right)\gamma_i\\
   &\ -([a_1]+a_1[a_2]+a_1a_2[a_3]+a_1a_2a_3[a_4]+a_1a_2a_3a_4[b_1])\otimes \left(\sum_{p=0}^{5}B^p\right),
  \end{aligned}
\]
where $A=\psi(b_1^{-1}a_4^{-1}a_3^{-1}a_2^{-1})$ and $B=A\circ \psi(a_1^{-1})$. 

It is straightforward to verify that 
\[\begin{aligned}
   A&=\begin{bmatrix}
                          2&-2&1&-1&1\\
                          1&0&0&-1&1\\
                          1&0&1&-2&1\\
                          1&0&1&-1&0\\
                          0&1&0&0&0
                         \end{bmatrix}\oplus I_{g-5},&
                         B&=\begin{bmatrix}
                          2&-2&1&-1&1\\
                          2&-1&0&-1&1\\
                          2&-1&1&-2&1\\
                          2&-1&1&-1&0\\
                          1&0&0&0&0
                         \end{bmatrix}\oplus I_{g-5}
  \end{aligned}
\]
and 
\[\sum_{p=0}^{9}A^p=10C,\quad\text{ }\quad \sum_{p=0}^{5}B^p=6C,\]
where
\[C=\begin{bmatrix}
     1&-1&1&-1&1\\
     1&-1&1&-1&1\\
     1&-1&1&-1&1\\
     1&-1&1&-1&1\\
     1&-1&1&-1&1\\
    \end{bmatrix}\oplus I_{g-5}.
\]
Hence this relation gives no new information.

Note that at this point we proved Theorem \ref{MainThm} for $N=N_{5,1}$,
\[\begin{aligned}H_1({\Cal{M}}(N_{5,1});H_1(N_{5,1};\zz))&= \gen{a_{1,3},u_{1,3},b_{1,1}-a_{1,1}-a_{3,3}}
\cong \zz_2\oplus\zz_2\oplus\zz_2.\end{aligned}\]
\subsection*{(A6)}
Relation (A6) gives
\[\begin{aligned}
   r^{(A6)}_{i}\!:0&=\left(\sum_{p=0}^{11}(a_2a_3a_4a_5a_6b_1)^p\right)([a_2]+a_2[a_3]+a_2a_3[a_4]+a_2a_3a_4[a_5]\\
   &\ +a_2a_3a_4a_5[a_6]+a_2a_3a_4a_5a_6[b_1])\otimes\gamma_i\\
   &\ -\left(\sum_{p=0}^{8}(a_1a_2a_3a_4a_5a_6b_1)^p\right)([a_1]+a_1[a_2]+a_1a_2[a_3]+a_1a_2a_3[a_4]\\
   &\quad\quad\quad+a_1a_2a_3a_4[a_5]+a_1a_2a_3a_4a_5[a_6]+a_1a_2a_3a_4a_5a_6[b_1])\otimes\gamma_i\\
   &=([a_2]+a_2[a_3]+\cdots+a_2a_3a_4a_5a_6[b_1])\otimes\left(\sum_{p=0}^{11}A^p\right)\gamma_i\\
   &\ -([a_1]+a_1[a_2]+\cdots+a_1a_2a_3a_4a_5a_6[b_1])\otimes \left(\sum_{p=0}^{8}B^p\right),
  \end{aligned}
\]
where $A=\psi(b_1^{-1}a_6^{-1}a_5^{-1}a_4^{-1}a_3^{-1}a_2^{-1})$ and $B=A\circ \psi(a_1^{-1})$. 

It is straightforward to verify that 
\[\begin{aligned}
   A&=\begin{bmatrix}
                          2&-2&1&-1&1&0&0\\
                          1&0&0&-1&1&0&0\\
                          1&0&1&-2&1&0&0\\
                          1&0&1&-1&0&0&0\\
                          0&2&0&0&0&-1&0\\
                          0&2&0&0&0&0&-1\\
                          0&1&0&0&0&0&0\\
                         \end{bmatrix}\oplus I_{g-7},\\ 
                         B&=\begin{bmatrix}
                          2&-2&1&-1&1&0&0\\
                          2&-1&0&-1&1&0&0\\
                          2&-1&1&-2&1&0&0\\
                          2&-1&1&-1&0&0&0\\
                          2&0&0&0&0&-1&0\\
                          2&0&0&0&0&0&-1\\
                          1&0&0&0&0&0&0\\
                         \end{bmatrix}\oplus I_{g-7}
  \end{aligned}
\]
and 
\[\sum_{p=0}^{11}A^p=12C,\quad\text{ }\quad \sum_{p=0}^{8}B^p=9C,\]
where
\[C=\begin{bmatrix}
     1&-1&1&-1&1&-1&1\\
     1&-1&1&-1&1&-1&1\\
     \cdots\\     
     1&-1&1&-1&1&-1&1\\
    \end{bmatrix}_{7\times 7}\oplus I_{g-7}.
\]
This implies that for $i\leq 7$ relation $r^{(A6)}_{i}$ is equivalent to
\[\begin{aligned}
   r^{(A6)}_{i}\!:0&=([a_1]+u_1[a_2]+u_1u_2[a_3]+u_1u_2u_3[a_4]+u_1u_2u_3u_4[a_5]\\
   &\ +u_1u_2u_3u_4u_5[a_6]+u_1u_2u_3u_4u_5u_6[b_1])\otimes(\gamma_1+\gamma_2+\cdots+\gamma_7)\\
   &=([a_1]+[a_2]+[a_3]+[a_4]+[a_5]+[a_6]+[b_1])\otimes(\gamma_1+\gamma_2+\cdots+\gamma_7)\\
   &\equiv 6\cdot 5 a_{1,3}+b_{1,5}+b_{1,6}+b_{1,7}\equiv a_{1,3}.
  \end{aligned}
\]
Hence generator (G1) is trivial for $g\geq 7$. For $i>7$ we get exactly the same conclusion as above.

Note that at this point we proved Theorem \ref{MainThm} for $N=N_{g,1}$, where $g\geq 7$ is odd,
\[\begin{aligned}&H_1({\Cal{M}}(N_{g,1});H_1(N_{g,1};\zz))=\gen{u_{1,3},b_{1,1}-a_{1,1}-a_{3,3}}=\zz_2\oplus\zz_2\quad\text{for $g\geq 7$ odd.}\end{aligned}\]
\subsection*{(A8) for $j=1$}
If $j=1$, relation (A8) gives
\[\begin{aligned}
   r^{(A8)}_{1:i}\!:0&=b_{2,i}+([a_1]+a_1[a_2]+a_1a_2[a_3]+a_1a_2a_3[a_4]+a_1a_2a_3a_4[a_5])\\
   &\quad \otimes(I_g+M+M^{2}+\cdots+M^{5})b_{2}^{-1}\gamma_i\\
&-([a_1]+a_1[a_2]+a_1a_2[a_3]+a_1a_2a_3[a_4]+a_1a_2a_3a_4[a_5]+a_1a_2a_3a_4a_5[b_1])\\
&\quad \otimes(I_g+N+N^{2}+\cdots N^{4})\gamma_i.
   \end{aligned}\]
Where $M=\psi (a_5^{-1}a_4^{-1}a_3^{-1}a_2^{-1}a_1^{-1})$ and $N=\psi(b_1^{-1})M$. Observe that we can reduce all coefficients in the above relation modulo 2, hence as matrices for $M$ and $N$ we can take
\[\begin{aligned}M&=\psi(u_5u_4u_3u_2u_1)\pmod 2=\begin{bmatrix}
              0&1&0&0&0&0\\
              0&0&1&0&0&0\\
              0&0&0&1&0&0\\
              0&0&0&0&1&0\\
              0&0&0&0&0&1\\
              1&0&0&0&0&0\\
             \end{bmatrix}\oplus I_{g-6},\\
  N&=\psi(b_1^{-1})M\pmod 2=\begin{bmatrix}
              0&0&1&1&1&0\\
              0&1&0&1&1&0\\
              0&1&1&0&1&0\\
              0&1&1&1&0&0\\
              0&0&0&0&0&1\\
              1&0&0&0&0&0\\
             \end{bmatrix}\oplus I_{g-6}.
\end{aligned}\]
We transform $r^{(A8)}_{1:i}$ further as follows.
\[\begin{aligned}
   r^{(A8)}_{1:i}\!:0&\equiv b_{2,i}+\big([a_1]\otimes(A-B)+[a_2]\otimes \psi(u_1)(A-B)+[a_3]\otimes\psi(u_2u_1)(A-B)\\
   &\quad +[a_4]\otimes \psi(u_3u_2u_1)(A-B)+[a_5]\otimes \psi(u_4u_3u_2u_1)(A-B)\big)\gamma_i\\
   &\quad-[b_1]\otimes\psi (u_5u_4u_3u_2u_1)B\gamma_i,
  \end{aligned}
\]
where
\[\begin{aligned}B&=\sum_{p=0}^4N^p\pmod 2=\begin{bmatrix}
                            1&0&1&0&1&0\\
                            0&1&0&1&0&1\\
                            1&0&1&0&1&0\\
                            0&1&0&1&0&1\\
                            1&0&1&0&1&0\\
                            0&1&0&1&0&1
                           \end{bmatrix}\oplus I_{g-6},\\
             A-B&=\sum_{p=0}^5M^p-B\pmod 2=\begin{bmatrix}
             0&1&0&1&0&1\\
                 1&0&1&0&1&0\\
                            0&1&0&1&0&1\\
                            1&0&1&0&1&0\\
                            0&1&0&1&0&1\\
                            1&0&1&0&1&0                            
                \end{bmatrix}\oplus I_{g-6}.
\end{aligned}\]
If $i>6$ this implies that generators (G8) are trivial for $j=2$, and if $i\leq 6$, we have
\[\begin{aligned}
   r^{(A8)}_{1:i}\!:
   0&\equiv   \begin{cases}
               b_{2,i}
               -b_{1,2}-b_{1,4}-b_{1,6}&\text{if $i\in\{1,3,5\}$}\\
               b_{2,i}               
               -b_{1,1}-b_{1,3}-b_{1,5}&\text{if $i\in\{2,4,6\}$}
              \end{cases}
   \\ &\equiv 
   \begin{cases}
	(b_{2,1}-a_{1,1}-a_{3,3}-a_{5,5})-b_{1,6}&\text{if $i=1$}\\
               (b_{2,i}-b_{2,1})+(b_{2,1}-a_{1,1}-a_{3,3}-a_{5,5})-b_{1,6}&\text{if $i\in\{3,5\}$}\\
               (b_{2,i}+b_{2,1})-(b_{2,1}-a_{1,1}-a_{3,3}-a_{5,5})-b_{1,5}&\text{if $i\in\{2,4,6\}$}.
              \end{cases}
  \end{aligned}
\]
Hence if $j=1$ then generators (G9)--(G10) are trivial, and generator (G8) is in the cyclic group of order 2 generated by generators (G1).
\subsection*{(A9a)}
Relation (A9a) gives
\[\begin{aligned}
   r^{(A9a)}_{i}\!:0&=[b_2]\otimes\left(I_g-\psi(b_1^{-1})\right)\gamma_i+[b_1]\otimes\left(\psi(b_2^{-1})-I_g\right)\gamma_i\\
   &\equiv\begin{cases}
      -(b_{2,4}+b_{2,1})-(b_{2,3}-b_{2,1})-(b_{2,2}+b_{2,1})&\text{if $i\leq 4$}\\
      0&\text{if $i>4$.}
     \end{cases}
  \end{aligned}
\]
This relation gives no new information.

Note that at this point we proved Theorem \ref{MainThm} for $N=N_{6,1}$,
\[\begin{aligned}H_1({\Cal{M}}(N_{6,1});H_1(N_{6,1};\zz))&\cong \gen{a_{1,3},u_{1,3},b_{1,1}-a_{1,1}-a_{3,3}}
\cong \zz_2\oplus\zz_2\oplus\zz_2.\end{aligned}\]
\subsection*{(A8) for $j>1$} We transform relation (A8) in exactly the same way, as in the case of $j=1$, however now we know that generators (G1) are trivial (since $g\geq 8$), hence we can completely ignore elements $a_{j,i}$. Moreover, we can inductively assume that generators (G8)--(G10) containing $b_{k,i}$ with $k<j+1$ are trivial, and generators (G11) containing $b_{k,1}$ with $k<j+1$ have order at most 2. Hence $r^{(A8)}_{j:i}$ takes form
\[
   r^{(A8)}_{j:i}\!:0\equiv b_{j+1,i}+[b_{j-1}]\otimes(A-B)\gamma_i-[b_j]\otimes\psi (u_{2j+3}u_{2j+2}u_{2j+1}u_{2j}b_{j-1}^{-1})B\gamma_i,  
\]
where $M=\psi (u_{2j+3}u_{2j+2}u_{2j+1}u_{2j}b_{j-1}^{-1})$,  $N=\psi(b_j^{-1})M$ and
\[A=\sum_{p=0}^5M^p,\ B=\sum_{p=0}^4N^p.\]
Now we will identify matrices for $M,N,A$ and $B$ -- the computations are straightforward, but a bit tedious. Using formulas for $u_j$ and $b_j$, we get 
\[\begin{aligned}
   M\gamma_i&\equiv \begin{cases}
   \gamma_i+(\gamma_1+\cdots+\gamma_{2j-1})+\gamma_{2j+4}&\text{if $i<2j$}\\
   \gamma_1+\cdots+\gamma_{2j-1}&\text{if $i=2j$}\\
   \gamma_{i-1}&\text{if $2j+1\leq i\leq 2j+4$}\\
      \gamma_i&\text{if $i>2j+4$}
     \end{cases}\\
   M^2\gamma_i&\equiv\begin{cases}
   \gamma_i+(\gamma_1+\cdots+\gamma_{2j-1})+\gamma_{2j+3}&\text{if $i<2j$}\\
   \gamma_{2j+4}&\text{if $i=2j$}\\
   \gamma_1+\cdots+\gamma_{2j-1}&\text{if $i=2j+1$}\\
   \gamma_{i-2}&\text{if $2j+2\leq i\leq 2j+4$}\\
      \gamma_i&\text{if $i>2j+4$}
     \end{cases}\\    
     M^4\gamma_i&\equiv\begin{cases}
   \gamma_i+(\gamma_1+\cdots+\gamma_{2j-1})+\gamma_{2j+1}&\text{if $i<2j$}\\
   \gamma_{i+2}&\text{if $2j\leq i\leq 2j+2$}\\
   \gamma_1+\cdots+\gamma_{2j-1}&\text{if $i=2j+3$}\\
   \gamma_{2j}&\text{if $i=2j+4$}\\
      \gamma_i&\text{if $i>2j+4$}
     \end{cases}\\
  \end{aligned}
\]
Hence
\[\begin{aligned}
   (I_g+M)\gamma_i&\equiv\begin{cases}
   (\gamma_1+\cdots+\gamma_{2j-1})+\gamma_{2j+4}&\text{if $i<2j$}\\
   \gamma_1+\cdots+\gamma_{2j-1}+\gamma_{2j}&\text{if $i=2j$}\\
   \gamma_{i-1}+\gamma_i&\text{if $2j+1\leq i\leq 2j+4$}\\
      0&\text{if $i>2j+4$}
     \end{cases}\\
     (I_g+M^2+M^4)\gamma_i&\equiv\begin{cases}
             \gamma_i+\gamma_{2j+1}+\gamma_{2j+3}&\text{if $i< 2j$}\\
             \gamma_{2j}+\gamma_{2j+2}+\gamma_{2j+4}&\text{if $i\in\{2j,2j+2,2j+4\}$}\\
     (\gamma_1+\cdots+\gamma_{2j-1})+\gamma_{2j+1}+\gamma_{2j+3}&\text{if $i\in\{2j+1,2j+3\}$}\\     
                 \gamma_i&\text{if $i>2j+4$}
            \end{cases}
            \end{aligned}
\]
\[
            A=(I_g+M)(I_g+M^2+M^4)\gamma_i\equiv\begin{cases}
                     \gamma_1+\cdots+\gamma_{2j+4}&\text{if $i\leq 2j+4$}\\
                      0&\text{if $i>2j+4$}
                    \end{cases}
\]
Obtained formula and our inductive assumption about generators (G8)--(G10) implies that we can ignore $A$ in the relation $r^{(A8)}_{j:i}$, hence it takes form
\[
   r^{(A8)}_{j:i}\!:0\equiv b_{j+1,i}+[b_{j-1}]\otimes B\gamma_i-[b_j]\otimes\psi (u_{2j+3}u_{2j+2}u_{2j+1}u_{2j}b_{j-1}^{-1})B\gamma_i.
\]
Now we concentrate on the formulas for $N=\psi(b_j^{-1})M$ and $B$.
\[\begin{aligned}
   N\gamma_i&\equiv\begin{cases}
                      \gamma_i+(\gamma_1+\cdots+\gamma_{2j-1})+\gamma_{2j+4}&\text{if $i<2j$}\\
                      \gamma_{2j}+\gamma_{2j+1}+\gamma_{2j+2}&\text{if $i=2j$}\\
   \gamma_{i-1}+(\gamma_1+\cdots+\gamma_{2j+2})&\text{if $2j+1\leq i\leq 2j+3$}\\
   \gamma_{2j+3}&\text{if $i=2j+4$}\\
                      \gamma_i&\text{if $i>2j+4$}
                     \end{cases}\\
                     N^2\gamma_i&\equiv\begin{cases}
          \gamma_i+(\gamma_1+\cdots+\gamma_{2j-1})+\gamma_{2j+3}&\text{if $i<2j$}\\
   \gamma_{2j+2}&\text{if $i=2j$}\\
                      \gamma_{2j}+\gamma_{2j+1}+\gamma_{2j+4}&\text{if $i=2j+1$}\\
   \gamma_{i-2}+(\gamma_1+\cdots+\gamma_{2j+1})+\gamma_{2j+4}&\text{if $i\in\{2j+2,2j+3\}$}\\
   (\gamma_1+\cdots+\gamma_{2j+1})&\text{if $i=2j+4$}\\
                      \gamma_i&\text{if $i>2j+4$}
                     \end{cases}
  \end{aligned}
\]
Hence
\[\begin{aligned}
   (N+N^2)\gamma_i&\equiv\begin{cases}
          \gamma_{2j+3}+\gamma_{2j+4}&\text{if $i<2j$}\\
          \gamma_{2j}+\gamma_{2j+1}&\text{if $i=2j$}\\
          (\gamma_1+\cdots+\gamma_{2j})+\gamma_{2j+2}+\gamma_{2j+4}&\text{if $i=2j+1$}\\
          \gamma_{2j}+\gamma_{2j+1}+\gamma_{2j+2}+\gamma_{2j+4}&\text{if $i=2j+2$}\\
          \gamma_{2j+1}+\gamma_{2j+4}&\text{if $i=2j+3$}\\
          (\gamma_1+\cdots+\gamma_{2j+1})+\gamma_{2j+3}&\text{if $i=2j+4$}\\
                      0&\text{if $i>2j+4$}
         \end{cases}\\
         (N+N^2)N^2\gamma_i&\equiv\begin{cases}
   \gamma_{2j+1}+\gamma_{2j+4}&\text{if $i<2j$}\\
   \gamma_{2j}+\gamma_{2j+1}+\gamma_{2j+2}+\gamma_{2j+4}&\text{if $i=2j$}\\
   \gamma_{2j}+\gamma_{2j+2}+\gamma_{2j+3}+\gamma_{2j+4}&\text{if $i=2j+1$}\\
   \gamma_{2j+1}+\gamma_{2j+2}&\text{if $i=2j+2$}\\
   (\gamma_1+\cdots+\gamma_{2j-1})+\gamma_{2j+4}&\text{if $i=2j+3$}\\   
   (\gamma_1+\cdots+\gamma_{2j+3})+\gamma_{2j}&\text{if $i=2j+4$}\\
               0&\text{if $i>2j+4$}.
              \end{cases}
  \end{aligned}
\]
Therefore $B=I_g+(N+N^2)+(N+N^2)N^2$ is equal to
\[\begin{aligned}
 B\gamma_i&\equiv\begin{cases}
                     \gamma_i+\gamma_{2j+1}+\gamma_{2j+3}&\text{if $i<2j$}\\
                     \gamma_{2j}+\gamma_{2j+2}+\gamma_{2j+4}&\text{if $i\in\{2j,2j+2,2j+4\}$}\\
                     (\gamma_1+\cdots+\gamma_{2j-1})+\gamma_{2j+1}+\gamma_{2j+3}&\text{if $i\in\{2j+1,2j+3\}$}\\
                     \gamma_i&\text{if $i>2j+4$,}
                    \end{cases}
\end{aligned}\]
and $B'=\psi (u_{2j+3}u_{2j+2}u_{2j+1}u_{2j}b_{j-1}^{-1})B$ is equal to
\[\begin{aligned}
 B'\gamma_i\equiv \begin{cases}
                     \gamma_i+(\gamma_1+\cdots+\gamma_{2j})+\gamma_{2j+2}+\gamma_{2j+4}&\text{if $i<2j$}\\
                     (\gamma_1+\cdots+\gamma_{2j-1})+\gamma_{2j+1}+\gamma_{2j+3}&\text{if $i\in\{2j,2j+2,2j+4\}$}\\
                     \gamma_{2j}+\gamma_{2j+2}+\gamma_{2j+4}&\text{if $i\in\{2j+1,2j+3\}$}\\
                     \gamma_i&\text{if $i>2j+4$.}  
                     \end{cases}
  \end{aligned}
\]
If $i>2j+4$ this implies that generators (G8) are trivial and if $i\leq 2j+4$ we get
\[\begin{aligned}
   r^{(A8)}_{j:i}\!:0&\equiv\begin{cases}
      b_{j+1,i}+b_{j-1,i}+b_{j-1,2j+1}+b_{j-1,2j+3}\\
      \quad -(b_{j,i}+b_{j,2j+2})-b_{j,2j+4}&\text{if $i<2j$}\\
      b_{j+1,i}+b_{j-1,2j}+b_{j-1,2j+2}+b_{j-1,2j+4}\\
      \quad -(b_{j,2j-1}+b_{j,2j+1})-b_{j,2j+3}&\text{if $i\in\{2j,2j+2,2j+4\}$}\\
      b_{j+1,i}+b_{j-1,2j-1}+b_{j-1,2j+1}+b_{j-1,2j+3}\\
      \quad -(b_{j,2j}+b_{j,2j+2})-b_{j,2j+4}&\text{if $i\in\{2j+1,2j+3\}$}      
     \end{cases}\\
     &\equiv\begin{cases}     
      b_{j+1,i}+b_{j-1,i}&\text{if $i<2j$}\\
      b_{j+1,i}+b_{j-1,2j}&\text{if $i\in\{2j,2j+2,2j+4\}$}\\
      b_{j+1,i}+b_{j-1,2j-1}&\text{if $i\in\{2j+1,2j+3\}$}
     \end{cases}\\
     &\equiv\begin{cases}     
      (b_{j+1,1}-a_{1,1}-\cdots-a_{2j+1,2j+1})\\
      \quad +(b_{j-1,1}-a_{1,1}-\cdots-a_{2j-3,2j-3})&\text{if $i=1$}\\
      (b_{j+1,i}\pm b_{j+1,1})\pm(b_{j+1,1}-a_{1,1}-\cdots-a_{2j+1,2j+1})\\
      \quad \pm (b_{j-1,1}-a_{1,1}-\cdots-a_{2j-3,2j-3})\pm(b_{j-1,2}+b_{j-1,1})&\text{if $1<i\leq 2j+4$.}
     \end{cases}
  \end{aligned}
\]
Hence we inductively proved that generators (G8)--(G10) are trivial, and generators (G11) are either trivial or they are in the cyclic group of order 2 generated by $b_{1,1}-a_{1,1}-a_{3,3}$.
\subsection*{(A9b)}
Relation (A9b) gives
\[\begin{aligned}
   r^{(A9b)}_{i}\!:0&=[b_{\frac{g-2}{2}}]\otimes\left(I_g-\psi(a_{g-5}^{-1})\right)\gamma_i+[a_{g-5}]\otimes\left(\psi\left(b_{\frac{g-2}{2}}^{-1}\right)-I_g\right)\gamma_i\\
   &\equiv \left[b_{\frac{g-2}{2}}\right]\otimes\left(I_g-\psi(u_{g-5})\right)\gamma_i\\   
   &\equiv\begin{cases}
      \left(b_{\frac{g-2}{2},g-4}+b_{\frac{g-2}{2},1}\right)-\left(b_{\frac{g-2}{2},g-5}-b_{\frac{g-2}{2},1}\right)&\text{if $i\in\{g-5,g-4\}$}\\
      0&\text{if $i\not\in\{g-5,g-4\}$.}
     \end{cases}
  \end{aligned}
\]
This relation gives no new information.

Note that at this point we proved that
\[\begin{aligned}&H_1({\Cal{M}}(N_{g,1});H_1(N_{g,1};\zz))
=\gen{u_{1,3},b_{1,1}-a_{1,1}-a_{3,3}}=
\zz_2\oplus\zz_2\quad\text{for $g\geq 8$ even,}\end{aligned}\]
which completes the proof of Theorem \ref{MainThm} for $N=N_{g,1}$.

Now we turn to the case of a closed surface $N=N_{g}$ where $g\geq 4$.
\subsection*{(B3)}
Relation (B3) has been carefully studied in Section 5.4 of \cite{PresHiperNon} (it is called (E3) there). If $g$ is even, it gives
\[r^{(B3)}_{i}\!:0=2a_{1,1}+2a_{3,3}+\cdots+2a_{g-1,g-1},\]
which means that generator (G12) is trivial, and if $g$ is odd, it gives
\[\begin{aligned}
   r^{(B3)}_{i}\!:0&=2a_{1,1}+2a_{3,3}+\cdots+2a_{g-2,g-2}-\ro_g\\
   &=-(2a_{g-1,g-1}+\ro_{g-1}+\ro_g)+(2a_{g-2,g-2}+\ro_{g-2}+\ro_{g-1})\\
   &\quad\cdots+(2a_{1,1}+\ro_1+\ro_2)+(2a_{2,2}+2a_{4,4}+\cdots+2a_{g-1,g-1})-\ro_1\\
   &\equiv (a_{1,1}+2a_{2,2}+2a_{4,4}+\cdots+2a_{g-1,g-1}-u_{1,1})-(a_{1,1}+\ro_1-u_{1,1}),
  \end{aligned}
\]
which means that generator (G12) is superfluous. 
\subsection*{(D1)}
Relation (D1) gives
\[\begin{aligned}
   r^{(D1)}_{j:i}\!:0&=([\ro]+\ro[a_j]-[a_j]-a_j[\ro])\otimes \gamma_i\\
   &=\begin{cases}
      -2a_{j,j}-\ro_j-\ro_{j+1}&\text{if $i=j$}\\
      (2a_{j,j}+\ro_j+\ro_{j+1})-2(a_{j,j}+a_{j,j+1})&\text{if $i=j+1$}\\
      -2a_{j,i}&\text{if $i\not\in\{j,j+1\}$.}
     \end{cases}
  \end{aligned}
\]
Hence generators (G6) are trivial.
\subsection*{(D2)}
Relation (D2) gives
\[\begin{aligned}
   r^{(D2)}_{j:i}\!:0&=([u_1]+u_1[\ro]+u_1\ro[u_1]-[\ro])\otimes\gamma_i\\
   &=\pm\begin{cases}
      (2a_{1,1}+\ro_1+\ro_2)-(u_{1,1}+u_{1,2})\\
      \quad -2(a_{1,1}+\ro_1-u_{1,1})&\text{if $i\in\{1,2\}$}\\
      0&\text{if $i>2$.}
     \end{cases}
  \end{aligned}
\]
Hence generator (G7) has order 2.
\subsection*{(E)}
Relation (E) gives no new information.
\[\begin{aligned}
   r^{(E)}_{j:i}\!:0&=([\ro]+\ro[\ro])\otimes \gamma_i=\ro_i-\ro_i=0.
  \end{aligned}
\]
\subsection*{(F)}
Before we deal with relation (F) let us try to establish some simplification rules, as in the case of relations not depending on  $g$. After rewriting relation (F) we will obtain a combination of generators (G1)--(G3) ((G3) only with $j=1$), (G6)--(G7) and (G12). Since generators (G2) and (G6) are trivial, we can assume that in the obtained relation there are no elements $a_{j,j+1}$ for $j=1,\ldots,g-1$ nor $\ro_j$ for $j\geq 2$. Using generator (G12) we can also remove all elements $a_{g-1,g-1}$, and since generators (G1) and (G3) generate cyclic groups, we can assume that in the obtained relation we have $a_{j,i}$ and $u_{j,i}$ only if $j=1$ and $i\in\{1,3\}$. As a result we will obtain a combination of $a_{1,3},u_{1,3}$ and generator (G7). All this elements have order 2, hence we can compute coefficients modulo 2. Moreover, the coefficient of generator (G7) is completely determined by coefficient of $\ro_1$, so we are not really interested in coefficients of $a_{1,1}$ and $u_{1,1}$.

For $n\in\{1,\ldots,g\}$ we define 
\[X_n=\psi(u_{n-1}\cdots u_2u_1\cdot u_1)=\begin{cases}
                                     \psi(u_1)&\text{if $n=1$}\\
                                     I_g&\text{if $n=2$}\\
                                     \psi(u_{n-1}\cdots u_2)&\text{if $n\in\{3,\ldots,g-1\}$}.
                                    \end{cases}
\] 
Relation (F) gives
\[\begin{aligned}
   r^{(F)}_{j:i}\!:0&=\left(\sum_{p=0}^{g-2}(u_1a_1a_2a_3\cdots a_{g-1}\ro)^{p}\right)[u_1]\otimes \gamma_i\\
   &\quad +\sum_{n=1}^{g-1}\left(\sum_{p=0}^{g-2}(u_1a_1a_2a_3\cdots a_{g-1}\ro)^{p}\right)(u_1a_1\cdots a_{n-1})[a_n]\otimes\gamma_i\\
   &\quad +\left(\sum_{p=1}^{g-1}(u_1a_1a_2a_3\cdots a_{g-1}\ro)^{p}\right)[\ro]\otimes\gamma_i\\
   &\equiv[u_1]\otimes \sum_{p=0}^{g-2}X_{g}^p\gamma_i+\sum_{n=1}^{g-1}[a_n]\otimes X_n\sum_{p=0}^{g-2}X_g^{p}\gamma_i+[\ro]\otimes \sum_{p=1}^{g-1}X_{g}^p\gamma_i.
  \end{aligned}
\]
Observe now that if $n\geq 3$ then $X_n$ acts on a basis $(\gamma_1,\ldots,\gamma_{g})$ as the cycle 
\[\gamma_n\mapsto \gamma_{n-1}\mapsto \ldots \mapsto \gamma_3\mapsto \gamma_2\mapsto \gamma_n\] 
of order $n-1$. In particular $X_g$ is the cycle 
\[\gamma_g\mapsto \gamma_{g-1}\mapsto \ldots \mapsto \gamma_3\mapsto \gamma_2\mapsto \gamma_g\] 
of order $g-1$, hence 
\[\sum_{p=0}^{g-2}X_g^{p}\gamma_i=\sum_{p=1}^{g-1}X_g^{p}\gamma_i=\begin{cases}
                                                                   (g-1)\gamma_1&\text{if $i=1$}\\
                                                                   \gamma_2+\cdots+\gamma_g&\text{if $i\in\{2,\ldots,g\}$.}
                                                                  \end{cases}
\]
Therefore for $i=1$ we have
\[\begin{aligned}
   r^{(F)}_{j:1}\!:0&=(g-1)\left(u_{1,1}+\sum_{n=1}^{g-1}[a_n]\otimes X_n\gamma_1+\ro_1\right)\\
   &=(g-1)(u_{1,1}+a_{1,1}+a_{2,1}+\cdots+a_{g-1,1}+\ro_1)\\
   &\equiv \begin{cases}
   0&\text{if $g$ odd}\\
            (a_{1,1}+\ro_1-u_{1,1})&\text{if $g$ even}.
           \end{cases}
  \end{aligned}
\]
For $i>1$ we have
\[\begin{aligned}
   r^{(F)}_{j:i}\!:0&=u_{1,2}+\cdots u_{1,g}+(a_{1,1}+a_{1,3}+\cdots+a_{1,g})+\\
   &\quad +\sum_{n=2}^{g-1}(a_{n,2}+\cdots+a_{n,g})+\ro_2+\cdots+\ro_g\\
   &\equiv\begin{cases}
      u_{1,3}+a_{1,3}+(a_{1,1}+2a_{2,2}+\cdots+2a_{g-1,g-1}-u_{1,1})&\text{if $g$ is odd}\\
      (a_{1,1}+\ro_1-u_{1,1})&\text{if $g$ is even.}
     \end{cases}
  \end{aligned}
\]
Hence in all cases generators (G7) and (G12) are superfluous.

This proves that
\[\begin{aligned}H_1({\Cal{M}}(N_{g});H_1(N_{g};\zz))&=\begin{cases}
                                      \gen{a_{1,3},u_{1,3},b_{1,1}-a_{1,1}-a_{3,3}}&\text{if $g\in\{4,5,6\}$}\\
                                      \gen{u_{1,3},b_{1,1}-a_{1,1}-a_{3,3}}&\text{if $g\geq 7$}
                                                           \end{cases}\\
                                                           &=\begin{cases}
                                      \zz_2\oplus\zz_2\oplus\zz_2&\text{if $g\in\{4,5,6\}$}\\
                                      \zz_2\oplus\zz_2&\text{if $g\geq 7$.}
                                                           \end{cases}\end{aligned}\]
which completes the proof of Theorem \ref{MainThm}. \nocite{GAP4}
\bibliographystyle{elsart-num-sort}

\end{document}